\documentclass[12pt,reqno]{amsart}
\usepackage{amsmath,amssymb,graphicx,epsfig,cite,rotating,setspace,amsthm,color}

\newcommand{\beq}{\begin{equation}}  
\newcommand{\eeq}{\end{equation}}  
\newcommand{\bea}{\begin{eqnarray}}  
\newcommand{\eea}{\end{eqnarray}}

\newcommand{\realpart}[1]{\operatorname{\sf Re}\!\left(#1\right)}

\newcommand{\eq}[2]{\begin{equation}\begin{split}#1\end{split}\label{#2}\end{equation}}

\newcommand{\inner}[2]{\left<#1,\,#2\right>}

\newcommand{\eqnn}[1]{\begin{equation}\begin{split}#1\end{split}\nonumber\end{equation}}

\newtheorem{theorem}{Theorem}
\numberwithin{theorem}{section}
\newtheorem{corollary}[theorem]{Corollary}

\newtheorem{lemma}[theorem]{Lemma}

\theoremstyle{definition}
\newtheorem{definition}[theorem]{Definition}
\newtheorem{remark}[theorem]{Remark}

\begin{document}
	
\title[Spectral instability of peakons]{Spectral instability of peakons in the $b$-family \\ of the Camassa--Holm equations}

\author[S. Lafortune]{St\'ephane Lafortune}
\address[S. Lafortune]{Department of Mathematics, College of Charleston, Charleston, SC 29401, USA}
\email{lafortunes@cofc.edu}

\author[D.E. Pelinovsky]{Dmitry E. Pelinovsky}
\address[D.E. Pelinovsky]{Department Of Mathematics \& Statistics,        %
McMaster University, Hamilton, Ontario, L8S 4K1, Canada}
\email{dmpeli@math.mcmaster.ca}

\date{}
\maketitle

\begin{abstract} 
	We prove spectral instability of peakons in the $b$-family of Camassa--Holm equations in $L^2(\mathbb{R})$ that includes the integrable cases of $b = 2$ and $b = 3$. We start with a linearized operator defined on functions in 
	$H^1(\mathbb{R}) \cap W^{1,\infty}(\mathbb{R})$ and extend it to a linearized operator defined on weaker functions in $L^2(\mathbb{R})$. For $b \neq \frac{5}{2}$, the spectrum of the linearized operator in $L^2(\mathbb{R})$ is proved to cover a closed vertical strip of the complex plane. For $b = \frac{5}{2}$, the strip shrinks to the imaginary axis, but an additional pair of real eigenvalues exists due to projections to the peakon and its spatial translation. The spectral instability results agree with the linear instability results in the case of the Camassa-Holm equation for $b = 2$. 
\end{abstract} 

\section{Introduction} 
\label{intro}

We consider the following $b$-family of Camassa--Holm equations (which we call $b$-CH),
\beq \label{bfamily} 
u_t - u_{xxt} +(b+1)uu_x=bu_x u_{xx} + u u_{xxx}. 
\eeq 
The family generalizes the classical cases of the Camassa--Holm (CH) equation for $b = 2$ and the Degasperis--Procesi (DP) equation for $b = 3$. 

The CH equation has first appeared in the study of bi-Hamiltonian structure 
of the Korteweg--de Vries (KdV) equation \cite{Fokas}. It was later introduced by Camassa and Holm \cite{Cam} in hydrodynamical
applications as a model for unidirectional wave propagation
on shallow water. The hydrodynamical relevance of the CH 
equation as a model for shallow water waves was discussed in 
\cite{Cam2,Johnson,Const4}. 

The DP equation can also be regarded as a model for nonlinear shallow water dynamics with its asymptotic accuracy equal to the CH equation \cite{dp}.  The $b$-CH family of equations was introduced in \cite{Dullin,dhh} by using transformations of the integrable KdV equation within the same asymptotic accuracy.

One of the most intriguing properties of the $b$-CH equations is
the occurrence of wave breaking when the solutions stay bounded 
but their slope develops a singularity in a finite time. 
Related to the wave breaking is the existence of peaked travelling waves 
called {\em peakons}. The exact peakon solution is given by 
\beq \label{peakon}
 u(x,t) = ce^{-|x-ct|} \qquad x \in \mathbb{R}.
\eeq
This solution is related to the reformulation 
of the $b$-CH equation \eqref{bfamily} in the weaker form
\begin{equation} 
\label{bfamilyi} \vspace{-.2cm}
u_t + u u_x + \frac{1}{4} \phi' \ast \left[ b u^2 + (3-b) u_x^2 \right] = 0,
\end{equation}
where $\ast$ denotes convolution and $\phi(x) = e^{-|x|}$ is Green's 
function satisfying $(1-\partial_x^2) \phi = 2 \delta_0$ 
with $\delta_0$ being the Dirac delta distribution centered at $x = 0$.
 
According to (\ref{peakon}), the exact peakon solution in the form 
$u(x,t) = c \phi(x-ct)$ satisfies the integral equation (\ref{bfamilyi}). 
We can assume without loss of generality that $c=1$ 
because the scaling transformation $u(x,t)\rightarrow a u(x,at)$ with arbitrary $a \in \mathbb{R}$ leaves \eqref{bfamilyi} invariant. It can then be checked directly 
that $\phi(x) = e^{-|x|}$ satisfies the integral equation
\begin{equation} \label{statCH}
-\phi +\frac{1}{2} \phi^2 +\frac{1}{4} \phi \ast \left[ b \phi^2 +(3-b) (\phi')^2 \right] = 0   
\end{equation}
piecewisely on both sides from the peak at $x = 0$, where the integral equation (\ref{statCH}) arises after substituting the travelling wave reduction 
to the $b$-CH equation (\ref{bfamilyi}) and integrating in $x$ with zero conditions at infinity.

Stability of peakons has been considered in the literature. Numerical simulations in \cite{Holm1,Holm2} showed that the peakons of the $b$-CH equation are likely to be unstable for $b < 1$. This conjecture was 
recently illustrated in  \cite{Char1} with the analysis of the linearized operator at the peakon solution and additional numerical experiments. 
In the case of $b<-1$, the numerical results of  \cite{Holm1,Holm2} suggested that arbitrary initial data moves to the left and asymptotically separates out into a number of smooth time-independent solitary waves. Smooth time-independent solitary wave solutions were shown to be orbitally stable for $b < -1$ in \cite{Hone14}. Smooth travelling solitary wave solutions in a regularized version of the b-CH equation were shown to be orbitally stable for $b = 2$ in \cite{CS-02} and for $b = 3$ in \cite{Liu-21}.

For $b > 1$, numerical simulations in \cite{Holm1,Holm2} showed that 
arbitrary initial data asymptotically resolves into a number of peakons. 
Orbital stability of peakons in the energy space $H^1(\mathbb{R})$ was shown 
for the CH equation $(b = 2)$ in \cite{Const5,Cons1} by using 
conservation of two energy integrals. This method was extended 
in \cite{LinLiu}, where the authors showed orbital stability of 
peakons for the DP equation ($b = 3$) in the energy space $L^2(\mathbb{R}) \cap L^3(\mathbb{R})$.
Since solutions of the initial-value problem 
for the $b$-CH equation with $b > 1$ are ill-posed in $H^s(\mathbb{R})$ for $s < \frac{3}{2}$ \cite{Him} (and in $H^{\frac{3}{2}}(\mathbb{R})$ for $1 < b \leq 3$ \cite{Molinet}) due to the lack of continuous dependence and norm inflation, smooth solutions to the CH and DP equations 
were considered in \cite{Cons1} and \cite{LinLiu} being close to the 
peakons in the energy space.

The largest class of initial data for which the initial-value problem 
is well-posed for the $b$-CH equation is given by the space $H^1(\mathbb{R}) \cap W^{1,\infty}(\mathbb{R})$ \cite{Linares}. 
It was recently shown in \cite{Natali} for the CH equation $(b=2)$ 
that although the peakons are orbitally stable in $H^1(\mathbb{R})$, they are unstable with respect to perturbations in $W^{1,\infty}(\mathbb{R})$ 
in the sense that the $W^{1,\infty}$ norm of the peaked perturbations may grow in time and may reach infinity in a finite time leading to the wave breaking of the solution. This analysis was performed 
by using the method of characteristics in the nonlinear evolution 
of the CH equation. 

Previous studies of stability avoid the question of the linearized 
stability of peakons because it was believed that ``the nonlinearity plays a dominant role rather than being a higher-order correction" and that ``the passage from the linear to the nonlinear theory is not an easy task, and may even be false" \cite{Cons1}. The first study of the linearized evolution 
of peaked perturbations in \cite{Natali} gave a valid evidence 
to this concern since it was found that the $H^1$ norm of the peaked perturbations grow within the linearized approximation exponentially as $e^{\frac{1}{2} t}$, while it does not grow in the full nonlinear evolution. 
It was recently clarified in \cite{MP-2021} in the setting of peaked periodic waves of the CH equation $(b=2)$ that both the growth of the $H^1$ norm of perturbations in the linearized approximation and its boundedness 
in the full nonlinear evolution of the CH equation are related to the same two conserved energies. 

These preliminary results raised an open question of whether the instability 
of the peakons can be understood from the spectral stability 
theory, where the instability of travelling waves follow
from the presence of the spectrum 
of a linearized operator in the right half plane of the complex plane. 
{\em The main purpose of this work is to give a definitive answer to this question with rigorous analysis of the spectral instability of peakons 
in the $b$-CH equation for any $b$.}

Since the local well-posedness of the initial-value problem 
holds in the space 
$H^1(\mathbb{R}) \cap W^{1,\infty}(\mathbb{R})$ 
which include peakons and their peaked perturbations \cite{Linares,Natali}, 
we first introduce the linearized operator acting on functions in this space. 
However, this space is restrictive for the spectral stability theory, 
hence we use projections to the peakon and its spatial translation 
in order to extend the linearized operator in $L^2(\mathbb{R})$ with a suitable 
defined domain, similarly to the recent studies of peaked periodic waves 
in the reduced Ostrovsky equation \cite{GP1,GP2}. 

We then analyze the spectrum of the linearized operator in $L^2(\mathbb{R})$. 
We prove that for $b \neq \frac{5}{2}$, the spectrum covers a closed vertical strip of the complex plane. The half-width of the strip is exactly $\frac{1}{2}$ 
for $b = 2$ which coincides with the exponential growth $e^{\frac{1}{2} t}$ of the $H^1$ norm of the perturbations  obtained in \cite{Natali}. We also note that the half-width of the strip for $b < \frac{5}{2}$ agrees with the observation made in 
Remark 3.6 in \cite{Char1} in their analysis of the differential operator 
at the peakon of the $b$-CH equation.  

For $b = \frac{5}{2}$, the strip shrinks to the imaginary axis, but we show that the projections to the peakon and its spatial translation also grow exponentially according to a system of two first-order differential equations. This additional instability suggests that the peakons of the $b$-CH equation (\ref{bfamilyi}) are spectrally unstable in $L^2(\mathbb{R})$ for every $b$.

The paper is organized as follows. Section \ref{main_res} explains the 
derivation of the linearized operator acting on functions in $H^1(\mathbb{R}) \cap W^{1,\infty}(\mathbb{R})$ and its extension to $L^2(\mathbb{R})$ with a suitable defined domain, 
after which the main result is formulated. Section \ref{sec-spectrum} gives the proof of the main result with the analysis of spectral properties 
of the linearized operator in $L^2(\mathbb{R})$. 
Section \ref{sec-time} describes the time evolution of the linearized 
equation in connection 
to the spectral properties of the linearized operator in $L^2(\mathbb{R})$.

\section{Linearized evolution} 
\label{main_res}

We recall from \cite{Natali} that if $u \in C([0,T),H^1(\mathbb{R}) \cap W^{1,\infty}(\mathbb{R}))$ is a weak solution to the $b$-CH equation (\ref{bfamilyi}) such that $u(t,\cdot + \xi(t)) \in C^1(-\infty,0) \cap C^1(0,\infty)$ for $t \in [0,T)$, then the single peak at $x = \xi(t)$ moves 
along the local characteristic curve with $\xi'(t) = u(t,\xi(t))$. 
Therefore, we decompose the solution near the peakon $\phi(x) = e^{-|x|}$ 
travelling with the unit speed into the following sum:
\begin{equation}
\label{decomp}
u(t,x) = \phi(x-t-a(t)) + v(t,x-t-a(t)), 
\end{equation}
where $v(t,\cdot) \in H^1(\mathbb{R}) \cap W^{1,\infty}(\mathbb{R})$ 
is the peaked perturbation such that $v(t,\cdot) \in C^1(-\infty,0) \cap C^1(0,\infty)$ 
and $a(t)$ is the deviation of the peak position of the perturbed peakon
from its unperturbed position satisfying $a'(t) = v(t,0)$. 
Substituting (\ref{decomp}) into (\ref{bfamilyi}) and using 
the stationary equation (\ref{statCH}) yields the following linearized equation:
\beq 
\label{eigp}
v_t = (1-\phi) v_{\xi} + (v_0 - v) \phi' 
-\frac{1}{2} \phi' \ast \left[ b \phi v + (3-b) \phi' v_{\xi} \right],
\eeq
where $v_0(t) := v(t,0)$ and $\xi := x - t -a(t)$ is the new spatial coordinate 
in the travelling frame. By using an elementary identity proven in \cite{Natali}, 
$$
\phi' \ast (\phi' v') = \phi \ast (\phi' v) - \phi' \ast (\phi v) + 2(v_0-v) \phi', \quad \forall v \in H^1(\mathbb{R}), 
$$
we can rewrite the linearized equation (\ref{eigp}) in the equivalent form:
\beq 
\label{eigp2}
v_t = (1-\phi)v_{\xi} + (b-2)(v_0-v)\phi' + Q(v),
\eeq
where
\beq
\label{Q-form1}
Q(v) := \frac{1}{2} (b-3) \phi \ast (\phi' v) - \frac{1}{2} (2b-3) \phi' \ast (\phi v).
\eeq
Using another elementary identity from \cite{Natali}, 
$$
\phi \ast (\phi' v) + \phi' \ast (\phi v) + 2 \phi v_{-1} = 0, \quad \forall v \in H^1(\mathbb{R}), \quad v_{-1}(\xi) := \int_0^{\xi} v(\xi') d\xi, 
$$
one can rewrite $Q(v)$ into the following two equivalent forms:
\beq
\label{Q-form2}
Q(v) = \frac{3}{2} (b-2) \phi \ast (\phi' v) + (2b-3) \phi v_{-1} 
= -\frac{3}{2} (b-2) \phi' \ast (\phi v) + (3-b) \phi v_{-1}.
\eeq
The following lemma shows that the linear operator $Q$ is compact in $L^2(\mathbb{R})$. 

\begin{lemma}
	\label{lem-compact}
	The operator $Q : L^2(\mathbb{R}) \mapsto L^2(\mathbb{R})$ is compact. 
\end{lemma}

\begin{proof}
Each term in either \eqref{Q-form1} or \eqref{Q-form2} can 
be written as an integral operator of the form
$$
\int_{-\infty}^\infty K(\xi,\xi') v(\xi')\, d\xi',
$$
for some kernel $K\in  L^2(\mathbb{R}^2)$. As such, each of those terms defines a Hilbert-Schmidt integral operator, known to be compact (see~\cite{Renardy}, p.~262). To be more specific, $\phi \ast (\phi' v)$ corresponds to the kernel
\beq\notag
	K_1=
	-{\mbox{sgn}}\!\left(\xi'\right)e^{-|\xi-\xi'|-|\xi'|},
\eeq
while $\phi v_{-1}$ corresponds to
\beq\notag
	K_2=
	\begin{cases}
		\displaystyle{{\mbox{sgn}}\!\left(\xi\right) e^{-|\xi|}\;\; \rm{for}\;\;0\leq|\xi'|\leq|\xi|},\\
		\displaystyle{0\;\;{\rm{otherwise}},}
	\end{cases}
\eeq
for which we obtain 
$$
\int_{-\infty}^{\infty} \int_{-\infty}^{\infty} |K_1(\xi,\xi')|^2 d \xi d \xi' 
= \int_{-\infty}^{\infty} \left( |\xi| + \frac{1}{2} \right) e^{-2|\xi|} d \xi = 1
$$
and
$$
\int_{-\infty}^{\infty} \int_{-\infty}^{\infty} |K_2(\xi,\xi')|^2 d \xi d \xi' 
= \int_{-\infty}^{\infty} |\xi| e^{-2|\xi|} d \xi = \frac{1}{2}.
$$
Hence, $K \in L^2(\mathbb{R}^2)$ and $Q$ is the compact Hilbert--Schmidt operator in $L^2(\mathbb{R})$.
\end{proof}

The linearized equation (\ref{eigp2}) with $Q(v)$ given by either (\ref{Q-form1}) or (\ref{Q-form2}) is well-defined 
in $v(t,\cdot) \in H^1(\mathbb{R}) \cap W^{1,\infty}(\mathbb{R})$. 
This can be shown by using either the local well-posedness theory \cite{Linares} or the method of characteristical curves \cite{Natali}. 
The linearized evolution depends on the value 
$v_0(t) = v(t,0)$ which is well-defined due to Sobolev embedding of 
$H^1(\mathbb{R})$ into the space of bounded and continuous functions. 

Next, we extend the linearized equation to the larger space ${\rm Dom}(L) \subset L^2(\mathbb{R})$ associated with the linearized operator 
\beq 
\label{Lop}
L := (1-\phi) \partial_{\xi} + (2-b) \phi' + Q.
\eeq
Since $Q$ is a compact operator in $L^2(\mathbb{R})$ by Lemma \ref{lem-compact} 
and $\phi \in W^{1,\infty}(\mathbb{R})$, the domain of $L$ 
in $L^2(\mathbb{R})$ is defined by 
\beq 
\label{Lop-domain}
{\mbox{Dom}}(L) = \left\{ v\in L^2(\mathbb{R}) : \quad (1-\phi) v' \in L^2(\mathbb{R}) \right\}.
\eeq
It follows from the bound $\| (1-\phi) v' \|_{L^2} \leq \| v' \|_{L^2}$ that $H^1(\mathbb{R})$ is continuously embedded into ${\rm Dom}(L)$. 
However, $H^1(\mathbb{R})$ is not equivalent to ${\rm Dom}(L)$ because 
$\phi' \in {\rm Dom}(L)$ but $\phi' \notin H^1(\mathbb{R})$ since $\xi \delta_0(\xi) \in L^2(\mathbb{R})$ but $\phi \notin C^1(\mathbb{R})$.  Generally, 
functions in ${\rm Dom}(L)$ do not have to be continuous across the peak at $x = 0$, therefore, $v_0$ may not be defined if $v \in {\rm Dom}(L)$ but $v \notin H^1(\mathbb{R})$. 

In order to extend the linearized equation (\ref{eigp2}) in ${\rm Dom}(L)$, we are going to use another equivalent reformulation of the linearized 
equation (\ref{eigp2}) for $v(t,\cdot) \in H^1(\mathbb{R}) \cap W^{1,\infty}(\mathbb{R})$. This is described in Lemma \ref{lem-reformulation} after the proof of the following elementary property of $L$.

\begin{lemma}
	\label{lem-projections}
For $L : {\rm Dom}(L) \subset L^2(\mathbb{R}) \mapsto L^2(\mathbb{R})$, it is true for every $b$ that 
\beq
\label{ppp}
		L\phi=(2-b)\phi'{\mbox{ and }}L\phi'=0.
\eeq	
\end{lemma}

\begin{proof}
By using the first expression in (\ref{Q-form2}), we obtain for every $\xi \neq 0$
$$
Q(\phi) = \frac{3}{2} (b-2) \phi \ast (\phi \phi') + (2b-3) \phi \int_0^{\xi} \phi(\xi') d\xi' = (1-b)( \phi' - \phi \phi')
$$
and
$$
Q(\phi') = \frac{3}{2} (b-2) \phi \ast (\phi' \phi') + (2b-3) \phi \int_0^{\xi} \phi'(\xi') d\xi' = -\phi + (b-1) \phi^2,
$$
which yields (\ref{ppp}) after substituting into (\ref{Lop}). Note that 
$\phi, \phi' \in {\rm Dom}(L)$ and that the singularity at $\xi = 0$ does not result in any contribution in $L^2(\mathbb{R})$ because $\xi \delta_0(\xi) = 0$ in $L^2(\mathbb{R})$.
\end{proof}

\begin{lemma}
	\label{lem-reformulation}
Consider the class of functions in 
$$
X := C(\mathbb{R},H^1(\mathbb{R}) \cap W^{1,\infty}(\mathbb{R})) \cap C^1(\mathbb{R},L^2(\mathbb{R}) \cap L^{\infty}(\mathbb{R})). 
$$
Then, $v \in X$ is a solution 
to the linearized equation (\ref{eigp2}) if and only if $\tilde{v} := 
v - v_0 \phi \in X$ satisfying $\tilde{v}(t,0) = 0$ is a solution of the linearized equation 
\beq 
\label{eigp3}
\tilde{v}_t = L \tilde{v} - \frac{3}{2} (b-2) \langle \phi \phi', \tilde{v} \rangle \phi,
\eeq
where the inner product $\langle \cdot,\cdot \rangle$ is defined in $L^2(\mathbb{R})$.
\end{lemma}

\begin{proof}
Substituting $v(t,x) = \tilde{v}(t,x) + v_0(t) \phi(x)$ into (\ref{eigp2}) yields 
$$
\tilde{v}_t + v_0'(t) \phi = L \tilde{v} + v_0 L \phi + (b-2) v_0 \phi'.
$$
The last two terms cancel out due to the first identity (\ref{ppp}) in Lemma \ref{lem-projections}. On the other hand, taking the limit $\xi \to 0$ 
into (\ref{eigp2}) in the class of functions 
$v \in X$ yields 
$$
v_0'(t) = \lim_{\xi \to 0} Q(v)(\xi) = \frac{3}{2} (b-2) \langle \phi \phi', v \rangle,
$$
where either representation in (\ref{Q-form2}) can be used together with the spatial symmetry of $\phi$. Since $\langle \phi \phi', \phi \rangle = 0$, 
it follows that $\langle \phi \phi', v \rangle = \langle \phi \phi', \tilde{v} \rangle$ so that the two equations yield (\ref{eigp3}). Since $v_0(t) = v(t,0)$, it is true that $\tilde{v}(t,0) = 0$. This constraint is preserved in the time evolution of (\ref{eigp3}) since $\lim\limits_{\xi \to 0} L \tilde{v} = \frac{3}{2} (b-2) \langle \phi \phi', \tilde{v} \rangle$ for $\tilde{v} \in X$.

The proof in the opposite direction from (\ref{eigp3}) to (\ref{eigp2}) is identical.
\end{proof}

The equivalent evolution problem (\ref{eigp3}) is still defined 
for $\tilde{v}(t,\cdot) \in H^1(\mathbb{R}) \cap W^{1,\infty}(\mathbb{R})$. However, the right-hand side is now well defined 
if $\tilde{v}(t,\cdot) \in {\rm Dom}(L) \subset L^2(\mathbb{R})$. This enables us to define linear stability of the peakons as follows.

\begin{definition}
	\label{def-instability}
The peakon solution $u(t,x) = \phi(x - t)$ of the $b$-CH equation \eqref{bfamilyi}  is said to be linearly stable if for every
$\tilde{v}_0 \in {\rm Dom}(L) \subset L^2(\mathbb{R})$, there exists a positive constant $C$ and a unique solution $\tilde{v} \in C(\mathbb{R},{\rm Dom}(L))$ to the linearized equation (\ref{eigp3}) with $\tilde{v}(0,\xi)=\tilde{v}_0(\xi)$ such that
$$
\|\tilde{v}(t,\cdot)\|_{L^2}\leq C \|{v_0} \|_{L^2}, \quad t>0.
$$
Otherwise, it is said to be linearly unstable.
\end{definition}

In order to prove that the peakons are linearly unstable in the sense of Definition \ref{def-instability} for all $b \in \mathbb{R}$, we reduce the linearized equation (\ref{eigp3}) in ${\rm Dom}(L) \subset L^2(\mathbb{R})$ to two parts, where one is defined by the linearized operator $L$ in (\ref{Lop})--(\ref{Lop-domain}) and the other one is defined by a system of two first-order differential equations. This task is achieved with the secondary decomposition described in the following lemma.

\begin{lemma}
	\label{lem-secondary}
	Consider the class of functions in 
	$$
	Y := C(\mathbb{R},{\rm Dom}(L)) \cap C^1(\mathbb{R},L^2(\mathbb{R})). 
	$$
	Then, $\tilde{v} \in Y$ is a solution 
	to the linearized equation (\ref{eigp3}) if $w := \tilde{v} - \alpha \phi - \beta \phi' \in Y$ is a solution of the linearized equation 
	\beq 
	\label{eigp4}
	\frac{dw}{dt} = L w,
	\eeq
	with $\alpha$ and $\beta$ satisfying the system 
	\beq
	\label{second-order}
	\frac{d \alpha}{dt} = (2-b) \beta + \frac{3}{2} (2-b) \langle \phi \phi', w\rangle, \quad \frac{d \beta}{dt} = (2-b) \alpha.
	\eeq
\end{lemma}

\begin{proof}
Substituting $\tilde{v}(t,x) = \alpha(t) \phi(x) + \beta(t) \phi'(x) + w(t,x)$ into (\ref{eigp3}) yields 
$$
\alpha'(t) \phi + \beta'(t) \phi' + w_t = (2-b) \alpha \phi' + L w +
(2-b) \beta \phi + \frac{3}{2} (2-b) \langle \phi \phi', w\rangle \phi,
$$
where we have used $\langle \phi, (\phi')^2 \rangle = \frac{2}{3}$, 
$\langle \phi \phi', \phi \rangle =0$, and the identities (\ref{ppp}) in Lemma \ref{lem-projections}. Separating $\phi$, $\phi'$ and the rest yields (\ref{eigp4}) and (\ref{second-order}). 
\end{proof}

\begin{corollary}
	\label{cor-stab}
	The peakon is linearly unstable in $Y$ for every $b \neq 2$.
\end{corollary}

\begin{proof}
	Setting $w = 0$ in (\ref{eigp4}) and (\ref{second-order}) gives the second-order homogeneous system
	$$
	\frac{d \alpha}{dt} = (2-b) \beta, \quad \frac{d \beta}{dt} = (2-b) \alpha,
	$$
	where the linear instability with the exponential growth $e^{|2-b|t}$ 
	exists for every $b \neq 2$. By Definition \ref{def-instability}, the peakon is linearly unstable in $Y$.
\end{proof}

\begin{remark}
	\label{remark-constraint}
In order to ensure the uniqueness of the decomposition $\tilde{v}(t,x) = \alpha(t) \phi(x) + \beta(t) \phi'(x) + w(t,x)$, we compute the adjoint 
operator $L^* : {\rm Dom}(L) \subset L^2(\mathbb{R}) \mapsto L^2(\mathbb{R})$ with respect to $\langle \cdot, \cdot \rangle$:
\beq
\label{Ls}
L^* v := (\phi-1) v_{\xi} + (3-b) \phi' v + \frac{1}{2} (b-3) \phi' (\phi \ast v) + \frac{1}{2} (2b-3) \phi (\phi' \ast v).
\eeq	
With straightforward computations, we obtain 
\begin{eqnarray*}
	&& L^* 1 = 0, \\
	&& L^* {\rm sgn} = 3 (b-2) \phi^2, \\
	&& L^* \phi^2 = 2(b-3) \phi \phi' + \frac{8}{3} (3-b) \phi^2 \phi', \\
	&& L^* \phi \phi' = (b-4) \phi^2 + \frac{8}{3} (3-b) \phi^3,
\end{eqnarray*}
where ${\rm sgn}(\xi) = 1$ if $\xi > 0$ and ${\rm sgn}(\xi) = -1$ if $\xi < 0$.

If $b = 2$, the constraints $\langle 1, w \rangle$ and $\langle {\rm sgn}, w \rangle$ are preserved in the time evolution of $w \in Y$. One can uniquely define $\alpha(t)$ and $\beta(t)$ by the orthogonality conditions 
 $\langle 1, w \rangle = 0$ and $\langle {\rm sgn}, w \rangle = 0$
so that $\alpha = \frac{1}{2} \langle 1, \tilde{v} \rangle$ and $\beta = -\frac{1}{2} \langle {\rm sgn}, \tilde{v} \rangle$, where we have used $\langle 1, \phi \rangle = -\langle {\rm sgn}, \phi' \rangle = 2$  and $\langle 1, \phi' \rangle = \langle {\rm sgn}, \phi \rangle = 0$.

If $b = 3$, the constraints $\langle \phi^2, w \rangle$ and 
$\langle {\rm sgn} + 3 \phi \phi', w \rangle$ are preserved in the time evolution of $w \in Y$. One can uniquely define $\alpha(t)$ and $\beta(t)$ by the orthogonality conditions 
$\langle \phi^2, w \rangle =\langle {\rm sgn} + 3 \phi \phi', w \rangle = 0$ and $\langle \phi^2, \phi' \rangle = \langle {\rm sgn} + 3 \phi \phi', \phi \rangle = 0$.

In Appendix \ref{AppA}, we derive an antisymmetric bounded function $v_b$ that satisfies $L^* v_b = 0$ for all $b > 3$ 
and thus also provides  a way to uniquely define $\alpha(t)$ and $\beta(t)$  by the orthogonality conditions 
$\langle 1, w \rangle =\langle v_b, w \rangle = 0$ and $\langle 1, \phi' \rangle = \langle v_b, \phi \rangle=0$ for $b > 3$.

In the general case, the decomposition $\tilde{v}(t,x) = \alpha(t) \phi(x) + \beta(t) \phi'(x) + w(t,x)$ is not uniquely defined so that the proof of Lemma \ref{lem-secondary} in the opposite direction from (\ref{eigp4}) and (\ref{second-order}) to (\ref{eigp3}) is incomplete. For instance, $w(t,x) = \phi(x) + (2-b) t \phi'(x)$ is a valid solution of the linearized equation (\ref{eigp4}) due to Lemma \ref{lem-projections}, although it is  redundant because system (\ref{second-order}) gives $\alpha(t) = -1$ and $\beta(t) = -(2-b)t$ that generates $\tilde{v}(t,x) = 0$ as a solution of (\ref{eigp3}).
\end{remark}

With the transformations of Lemma \ref{lem-reformulation} and \ref{lem-secondary}, we have reduced the linearized evolution 
(\ref{eigp2}) defined in $X$ to the linearized evolution 
(\ref{eigp4}) defined in $Y$. As a result, the spectral 
properties of the operator $L$ in (\ref{Lop})--(\ref{Lop-domain}) 
determine the linear stability of the peakons in addition 
to the result of Corollary \ref{cor-stab}. 

The main result of this paper is about the spectrum 
of the operator $L$ according to the following standard definition 
(see Definition 6.1.9 in \cite{Buhler18}).

\begin{definition}
	Let $A$ be a linear operator on a Banach space $X$ with ${\rm Dom}(A) \subset X$.	The complex plane $\mathbb{C}$ is decomposed into the following two sets:
	\begin{enumerate}
		\item The resolvent set
		$$
		\rho(A) =  \left\{\lambda \in \mathbb{C} : \;\;
		{\rm Ker}(A-\lambda I) = \{0\}, \;\; {\rm Ran}(A-\lambda I) = X, \;\;  (A-\lambda I)^{-1}: X \to X \text{ is bounded} \right\}.
		$$
		\item The spectrum
		$$
		\sigma(A) =\mathbb{C} \setminus \rho(A),
		$$
		which is further decomposed into the following three disjoint sets:
		\begin{enumerate}
			\item the point spectrum
			$$
			\sigma_{\rm p}(A) = \{ \lambda \in \sigma(A) : \;\; {\rm Ker}(A-\lambda I) \neq \{0\}\},
			$$
			\item the residual spectrum
			$$
			\sigma_{\rm r}(A) = \{ \lambda \in \sigma(A) : \;\; {\rm Ker}(A-\lambda I) = \{0\}, \;\; {\rm Ran}(A-\lambda I) \neq X \},
			$$
			\item the continuous spectrum
			$$
			\sigma_{\rm c}(A) = \{ \lambda \in \sigma(A) : \;\; {\rm Ker}(A-\lambda I) = \{0\}, \;\; {\rm Ran}(A-\lambda I) = X, \;\;
			(A-\lambda I)^{-1} : X \to X \text{ is unbounded}\}.
			$$
		\end{enumerate}\medskip
	\end{enumerate}
	\label{def-spectrum}
\end{definition}

The following theorem represents the main result of this paper. 

\begin{theorem}
	\label{S}
	The spectrum of the linear operator $L$ defined by (\ref{Lop})--(\ref{Lop-domain}) is given by  
	\begin{equation}
	\notag
	\sigma(L) = \left\{ \lambda \in\mathbb{C} : \;\; |\realpart{\lambda}|\leq  \left|\frac{5}{2}-b\right| \right\}.
	\end{equation}
Moreover, the point spectrum 
is located for $0 < |\realpart{\lambda}| <  \frac{5}{2}-b$ 
if $b < \frac{5}{2}$ and the residual spectrum is located 
for $0 < |\realpart{\lambda}| <  b- \frac{5}{2}$ 
if $b > \frac{5}{2}$, whereas the continuous spectrum is located for 
$\realpart{\lambda} = 0$ and $\realpart{\lambda} = \pm \left|\frac{5}{2}-b\right|$
in both cases. Additionally, $\lambda=0$ is the eigenvalue of the point spectrum of algebraic multiplicity 2 embedded into the continuous spectrum for every $b$.
\end{theorem} 

\begin{corollary}
	\label{cor-instab}
	The peakon is linearly unstable in $Y$ for every $b \neq \frac{5}{2}$.
\end{corollary}

\begin{proof}
	Recall the balance equation, which follows from the linearized equation (\ref{eigp4}) for $w \in Y$:
\beq
\label{balance-eq}
	 {{\frac{1}{2}\frac{d}{dt}\left( \| w \|_{L^2}^2\right)}} = \langle L w, w \rangle = \langle w, L^* w \rangle.
\eeq
	If $\lambda \in (0,\frac{5}{2}-b)$ is a real eigenvalue for the point spectrum of $L$ for $b < \frac{5}{2}$ and $w \in {\rm Dom}(L) \subset L^2(\mathbb{R})$ is the corresponding eigenfunction of $L$, then the balance equation (\ref{balance-eq}) suggests exponential growth of $\| w \|_{L^2}$ with the rate $e^{\lambda t}$. 	If $\lambda \in (0,b-\frac{5}{2})$ is a real eigenvalue for the residual spectrum of $L$ for $b > \frac{5}{2}$, then $\lambda$ is the eigenvalue for the point spectrum of $L^*$ since $\sigma_r(L) = \sigma_p(L^*)$ if $\sigma_p(L)$ is an empty set by Lemma 6.2.6 in \cite{Buhler18}. Let $w \in {\rm Dom}(L) \subset L^2(\mathbb{R})$ be an eigenfunction of $L^*$ for the eigenvalue $\lambda$. The second equality in the balance equation (\ref{balance-eq}) suggests exponential growth of $\| w \|_{L^2}$ with the rate $e^{\lambda t}$. In both cases, 
 	the linear evolution of $w \in Y$ in  grows exponentially 
	in the $L^2$ norm if $b \neq \frac{5}{2}$. By Definition \ref{def-instability}, the peakon is linearly unstable in $Y$.
\end{proof}

\begin{remark}
	Since linear instabilities of Corollaries \ref{cor-stab} and \ref{cor-instab} vanish at different values of $b$, the peakon is linearly unstable in $Y$ for every $b$.
\end{remark}

\begin{remark}
	For each exponentially growing solution $w \in Y$ of the linearized equation (\ref{eigp4}), 	one can find the unique solution of the second-order system (\ref{second-order}) 
	which grows either exponentially with the same rate if $|b-2| > |\frac{5}{2} - b|$ (or $b > \frac{9}{4}$) (when the unstable eigenvalue $|b-2|$ is outside the strip in Theorem \ref{S}) or exponentially times polynomially 
	if $|b-2| \leq |\frac{5}{2} - b|$ (or $b \leq \frac{9}{4}$) (when the unstable eigenvalue $|b-2|$ is inside the strip in Theorem \ref{S}). 
\end{remark}

\section{Spectrum of $L$}
\label{sec-spectrum}

We decompose $L$ given by \eqref{Lop}--\eqref{Lop-domain} as
\beq\notag
	L = L_0 +Q,
\eeq
where $L_0 : {\rm Dom}(L) \subset L^2(\mathbb{R}) \mapsto L^2(\mathbb{R})$ is given by 
\beq
\label{L0}
	L_0 := (1-\phi) \partial_{\xi} +(2-b)\phi'
\eeq
and $Q : L^2(\mathbb{R}) \mapsto L^2(\mathbb{R})$ is the compact operator by Lemma \ref{lem-compact}. By Theorem 1 in \cite{GP2}, if the intersections $\sigma_p(L_0) \cap \rho(L)$ and $\sigma_p(L) \cap \rho(L_0)$ are empty, 
then $\sigma(L) = \sigma(L_0)$. The proof of Theorem \ref{S} is achieved by computing the spectrum of $L_0$ and the point spectrum $L$.

\subsection{Spectrum of $L_0$}

The spectrum of $L_0$ is described by the following theorem.

\begin{theorem}
	\label{S0}
	The spectrum of the linear operator $L_0 : {\rm Dom}(L) \subset L^2(\mathbb{R}) \mapsto L^2(\mathbb{R})$ defined by (\ref{L0}) is given by 
	\begin{equation}\notag
	\sigma(L_0) = \left\{ \lambda \in\mathbb{C} : \;\; |\realpart{\lambda}|\leq  \left|\frac{5}{2}-b\right| \right\}.
	\end{equation}
	Moreover, the point spectrum 
	is located for $0 < |\realpart{\lambda}| <  \frac{5}{2}-b$ 
	if $b < \frac{5}{2}$ and the residual spectrum is located 
	for $0 < |\realpart{\lambda}| <  b- \frac{5}{2}$ 
	if $b > \frac{5}{2}$, whereas the continuous spectrum is located for 
	$\realpart{\lambda} = 0$ and $\realpart{\lambda} = \pm \left|\frac{5}{2}-b\right|$
	in both cases. 
\end{theorem} 

\begin{proof}
	Given simplicity of the definition of $L_0$, the proof is obtained by computing the point, residual, and continuous spectrum of $L_0$ explicitly. \\
	
	{\em Point spectrum of $L_0$.}  We solve the differential equation 
\beq
\label{L0E}
		(1-\phi) \frac{dv}{d\xi} + (2-b) \phi' v = \lambda v, \quad \xi \in \mathbb{R}.
\eeq
The differential equation (\ref{L0E}) is solved separately for $\xi > 0$ and $\xi < 0$ with the following general solution:
\beq
\label{Lv0S}
		v(\xi) =
		\left\{ 
		\begin{array}{ll}
			\displaystyle{v_+\frac{(e^\xi-1)^\lambda}{(1-e^{-\xi})^{b-2}}}, & \quad  \xi>0,\\
			\displaystyle{v_-\frac{1}{(e^{-\xi}-1)^\lambda{(1-e^{\xi})^{b-2}}}}, & \quad \xi<0, \end{array}\right.
\eeq
	where $v_-$ and $v_+$ are arbitrary constants. 
	
The differential equation (\ref{L0E}) has the following symmetry:	
if $\lambda=\lambda_0$ is an eigenvalue with the eigenfunction $v = v_0(\xi)$, then $\lambda=-\lambda_0$ is an eigenvalue with the eigenfunction $v = v_0(-\xi)$. Therefore, it is sufficient to consider the case of $\realpart{\lambda} \geq 0$.

Since $v(\xi) \sim v_+ e^{\lambda \xi}$ as $\xi \to +\infty$, then 
$v \in L^2(\mathbb{R})$ for $\realpart{\lambda} \geq 0$ is satisfied by the only choice of $v_+ = 0$. Since $v(\xi) \sim v_- e^{\lambda \xi}$ as $\xi \to -\infty$, then $v \in L^2(\mathbb{R})$ is satisfied for arbitrary $v_-$ 
if and only if $\realpart{\lambda} > 0$. Since $v(\xi) \sim v_- |\xi|^{2-b-\lambda}$ as $x \to 0^-$, then $v \in L^2(\mathbb{R})$ is satisfied for arbitrary $v_-$ if and only if $\realpart{\lambda}+b-2< \frac{1}{2}$, 
that is for $\realpart{\lambda} < \frac{5}{2} - b$. In all other cases, we have to set $v_- = 0$ for $v \in L^2(\mathbb{R})$ so that the zero function is the only solution of (\ref{L0E}) in $L^2(\mathbb{R})$.  
Summarizing and using the symmetry above, $\sigma_p(L_0)$ exists if $b < \frac{5}{2}$ and is located for $0 < |\realpart{\lambda}| <  \frac{5}{2}-b$. \\

{\em Residual spectrum of $L_0$.}  By Lemma 6.2.6 in \cite{Buhler18}, 
if $\sigma_p(L_0)$ is an empty set, then $\sigma_r(L_0) = \sigma_p(L_0^*)$, 
where 
\beq
\notag
L_0^* = - \partial_{\xi} (1-\phi) + (2-b) \phi' = -(1-\phi) \partial_{\xi} + (3-b) \phi'
\eeq 
is the adjoint operator to $L_0$ in $L^2(\mathbb{R})$. The differential equation 
\beq
\label{EPs}
-(1-\phi) \frac{dv}{d\xi} + (3-b) \phi' v = \lambda v, \quad \xi \in \mathbb{R}
\eeq
becomes (\ref{L0E}) after the transformation: $\lambda \mapsto -\lambda$ and 
$b-2 \mapsto 3-b$. Therefore, we obtain the following general solution
by applying this transformation to (\ref{Lv0S}):
\beq
\label{Lv0A}
v(\xi) =
\left\{ 
\begin{array}{ll}
	\displaystyle{v_+\frac{(e^\xi-1)^{-\lambda}}{(1-e^{-\xi})^{3-b}}}, & \quad  \xi>0,\\
	\displaystyle{v_-\frac{1}{(e^{-\xi}-1)^{-\lambda}{(1-e^{\xi})^{3-b}}}}, & \quad \xi<0, \end{array}\right.
\eeq
where $v_-$ and $v_+$ are arbitrary constants. Proceeding similarly with the limits $\xi \to \pm \infty$ and $\xi \to 0^{\pm}$ shows 
that nonzero solutions of (\ref{EPs}) exist in $L^2(\mathbb{R})$ 
if and only if $b > \frac{5}{2}$ and $0 < |\realpart{\lambda}| <  \frac{5}{2}-b$. This yields $\sigma_r(L_0)$ since $\sigma_p(L_0)$ is an empty set for $b > \frac{5}{2}$. \\

{\em Resolvent set of $L_0$.} Let us consider the resolvent equation
\beq
\label{L0RE}
		L_0v-\lambda v=f,
\eeq
where $f \in L^2(\mathbb{R})$ is arbitrary and $\realpart{\lambda} \geq 0$ is assumed without loss of generality. We multiply both sides of \eqref{L0RE} by $\bar{v}$ and integrate over $\mathbb{R}$. Using the definition of $L_0$ given in 
	\eqref{L0}, one finds
	\eq{
		\inner{\left((1-\phi)v\right)'}{v}+(3-b)\inner{\phi' v}{v}-\lambda\|v\|^2=\inner{f}{v}.
	}{in}
By integration by parts, since $\lim\limits_{\xi \to \pm \infty} v(\xi) = 0$ for $v \in {\rm Dom}(L)$, we 
	have that
	$$
	\inner{\left((1-\phi) v\right)'}{v}=-\inner{v}{\left((1-\phi) v\right)'}-\inner{\phi' v}{v},
	$$
	thus
\begin{equation}
\label{tech-eq}
	\realpart{\inner{((1-\phi)v)'}{v}}=-\frac{1}{2}\inner{\phi' v}{\bar{v}}.
\end{equation}
	Taking the real part of \eqref{in}, multiplying by -1, and using \eqref{tech-eq}, we get 
\beq
\label{inr}
		\left( b -\frac{5}{2}\right) \inner{\phi' v}{v}+\realpart{\lambda}\|v\|^2=-\realpart{\inner{f}{v}},
\eeq
where
\beq
\label{inu}
	-\|v\|^2 \leq \inner{\phi' v}{v}\leq \|v\|^2.
\eeq
Using the upper bound of (\ref{inu}) in \eqref{inr} in the case $b \leq \frac{5}{2}$, we find that 
\beq
\notag
		\left(\realpart{\lambda}+b-\frac{5}{2}\right)\|v\|^2\leq\left|\realpart{\inner{f}{v}}\right| \leq \| f \| \| v \|,
\eeq
where the Cauchy--Schwarz inequality has been used. Hence, for every $\realpart{\lambda}> \frac{5}{2}-b$, there exists $C_{\lambda}$ such that 
$\| v \| \leq C_{\lambda} \|f \|$ so that this $\lambda$ belongs to $\rho(L_0)$.

In a very similar way, using the lower bound of (\ref{inu}) in \eqref{inr} 
in the case $b\geq \frac{5}{2}$, we find that 
\eqnn{
		\left(\realpart{\lambda}-b+\frac{5}{2}\right)\|v\|^2 \leq\left|\realpart{\inner{f}{v}}\right| \leq \| f \| \| v \|,
	}
hence $\realpart{\lambda} > b- \frac{5}{2}$ belongs to $\rho(L_0)$. \\
	
{\em Continuous spectrum of $L_0$.}  By Theorem 4 in \cite[p. 1438]{DS}, if $L_0$ is a differential operator on $\mathbb{R} = (-\infty,0) \cup (0,\infty)$, 
and $L_0^{\pm}$ are restrictions of $L_0$ on $(-\infty,0)$ and $(0,\infty)$, 
then $\sigma_c(L_0) = \sigma_c(L_0^+) \cup \sigma_c(L_0^-)$. Therefore, we can represent the resolvent equation (\ref{L0RE}) separately for $\xi > 0$ and $\xi < 0$.

For $\xi > 0$, we use the transformation $(0,\infty) \ni \xi \mapsto z \in \mathbb{R}$ by $z = \log(e^{\xi}-1)$ which follows from the method of characteristics. Then, $\mathfrak{v}(z) := v(\xi)$ satisfies the resolvent equation 
\beq
\notag\mathfrak{L}_0^+ \mathfrak{v} - \lambda \mathfrak{v} = \mathfrak{f},
\eeq 
where $\mathfrak{L}_0^+ := \partial_z + (b-2)(1+e^z)^{-1}$ and
$\mathfrak{f}(z) := f(\xi)$. Since
$$
\int_0^{\infty} v(\xi)^2 d\xi = \int_{-\infty}^{0} e^{z} \frac{\mathfrak{v}(z)^2 dz}{1 + e^{z}} + \int_{0}^{\infty} \frac{\mathfrak{v}(z)^2 dz}{1 + e^{-z}}, 
$$
$\sigma_c(L_0^+)$ is the union of $\sigma_c(\mathfrak{L}_0^+)$ in $L^2(\mathbb{R}_+)$ and $L^2(\mathbb{R}_-;e^z)$, where $e^z$ is the exponential weighted $L^2$ space. Since $\mathfrak{L}_0^+ = \partial_z + e^{-z} (b-2) (1+e^{-z})^{-1}$ and $e^{-z} (1+e^{-z})^{-1}$ is bounded and decaying exponentially to $0$ as $z \to +\infty$, $\sigma_c(\mathfrak{L}_0^+)$ in $L^2(\mathbb{R}_+)$ is located at $\realpart{\lambda} = 0$. Since 
$\mathfrak{L}_0^+ = \partial_z + (b-2) - (b-2) e^z (1+e^{z})^{-1}$ and $e^z (1+e^z)^{-1}$ is bounded and decaying exponentially to $0$ as $z \to -\infty$,
$\sigma_c(\mathfrak{L}_0^+)$  in $L^2(\mathbb{R}_-;e^z)$ is located at 
$\realpart{\lambda} = b-\frac{5}{2}$.

Similarly, for $\xi < 0$, we use the transformation $(-\infty,0) \ni \xi \mapsto z \in \mathbb{R}$ by $z = -\log(e^{-\xi}-1)$ and obtain the resolvent equation for $\mathfrak{v}(z) := v(\xi)$:
\beq
\notag
\mathfrak{L}_0^- \mathfrak{v} - \lambda \mathfrak{v} = \mathfrak{f},
\eeq 
where $\mathfrak{L}_0^- := \partial_z + (2-b)(1+e^{-z})^{-1}$ and
$\mathfrak{f}(z) := f(\xi)$. Since
$$
\int_{-\infty}^0 v(\xi)^2 d\xi = \int_{-\infty}^{0} \frac{\mathfrak{v}(z)^2 dz}{1 + e^{z}} + \int_{0}^{\infty} e^{-z} \frac{\mathfrak{v}(z)^2 dz}{1 + e^{-z}}, 
$$
$\sigma_c(L_0^-)$ is the union of $\sigma_c(\mathfrak{L}_0^-)$ in $L^2(\mathbb{R}_-)$ and 
$L^2(\mathbb{R}_+;e^{-z})$. Similar to the previous arguments, 
the continuous spectrum of $\mathfrak{L}_0^- = \partial_z + e^{z} (2-b) (1+e^{z})^{-1}$ in $L^2(\mathbb{R}_-)$ is located 
at $\realpart{\lambda} = 0$ and the continuous spectrum of 
$\mathfrak{L}_0^- = \partial_z + (2-b) - (2-b) e^{-z} (1+e^{-z})^{-1}$ in $L^2(\mathbb{R}_+;e^{-z})$ is located at 
$\realpart{\lambda} = \frac{5}{2}-b$. 

By Theorem 4 in \cite{DS}, $\sigma_c(L_0)$ is located for 
$\realpart{\lambda} = 0$ and $\realpart{\lambda} = \pm \left|\frac{5}{2}-b\right|$.
\end{proof}

\subsection{Point spectrum of $L$}

For the point spectrum of $L$, we consider the spectral problem
\beq
	Lv - \lambda v=0, \quad v \in {\rm Dom}(L) \subset L^2(\mathbb{R}).
\label{eig}
\eeq
By Lemma \ref{lem-projections}, $0$ is always a double eigenvalue associated with the two-dimensional invariant subspace $\{ \phi, \phi'\}$. If 
$b \neq 2$, the double zero eigenvalue is defective with only one linearly independent eigenfunction. For $b = 2$, the double zero eigenvalue is 
semi-simple. In what follows, we are looking for other solutions of 
the spectral problem (\ref{eig}). 

\begin{remark}
	The double zero eigenvalue of $L$ is related to the secondary decomposition $\tilde{v} = \alpha \phi + \beta \phi' + w \in Y$ in Lemma \ref{lem-secondary}. By Remark \ref{remark-constraint}, the decomposition is not unique in general. It is therefore unclear how to set up a constrained subspace of $L^2(\mathbb{R})$ so that the double zero eigenvalue of $L$ is eliminated by the constraints.
\end{remark}

The following lemma characterizes the point spectrum of $L$.
	
\begin{lemma}
\label{im}
In addition to the double zero eigenvalue, the linear operator  $L : {\rm Dom}(L) \subset L^2(\mathbb{R}) \mapsto L^2(\mathbb{R})$ defined by \eqref{Lop} admits the point spectrum for $0 < |\realpart{\lambda}| < \frac{5}{2}-b$ if $b < \frac{5}{2}$.
\end{lemma} 

\begin{remark}
Some ideas of the proof of Lemma \ref{im} are similar to what was done recently in \cite{Char1}. However, the space of functions was different and the direct linearization of the $b$-CH equation was used in \cite{Char1} without appealing to the decomposition (\ref{decomp}). In additioin, we find all the solutions for the spectral problem (\ref{eig}) explicitly.
\end{remark}

\begin{proof}
The spectral problem (\ref{eig}) has the same symmetry 
as the differential equation (\ref{L0E}): if $\lambda=\lambda_0$ is an eigenvalue with the eigenfunction $v = v_0(\xi)$, then $\lambda=-\lambda_0$ is an eigenvalue with the eigenfunction $v = v_0(-\xi)$. Therefore, it is sufficient to consider the case $\realpart{\lambda} \geq 0$.

Applying the operator $1-\partial_\xi^2$ to (\ref{eig}) separately for $\xi < 0$ and $\xi> 0$ yields the following differential equation
\beq 
\label{left}
\lambda (v-v'') = (1-\phi)(v'-v''') - b \phi' (v-v'').
\eeq 
Indeed, for $\xi < 0$, we compute from the representation (\ref{Q-form2})
\begin{eqnarray*}
(1-\partial_{\xi}^2) \phi \ast (\phi' v) =
(1-\partial_{\xi}^2) \left[ \int_{-\infty}^{\xi} e^{-\xi+2\eta} v(\eta) d\eta + \int_{\xi}^0 e^{\xi} v(\eta) d\eta + \int_0^{\infty} e^{\xi - 2 \eta} v(\eta) d \eta \right] = 2 e^{\xi} v 
\end{eqnarray*}
and 
\begin{eqnarray*}
	(1-\partial_{\xi}^2) \phi v_{-1} =
	(1-\partial_{\xi}^2) e^{\xi} \int_{0}^{\xi} v(\eta) d\eta = -2 e^{\xi} v -e^{\xi} v',
\end{eqnarray*}
which yields (\ref{left}) for $\xi < 0$ after adding all terms. Computations for $\xi > 0$ are similar. 

Let $m := v - v''$. The differential equation (\ref{left}) becomes the first-order equation
\beq 
\label{right}
(1-\phi) \frac{dm}{d\xi}  - b \phi' m = \lambda m,
\eeq 
which coincides with the differential equation (\ref{L0E}) after replacing $b$ by $b-2$. Therefore, the exact solution is given by the corresponding transformation of the solution (\ref{Lv0S}), namely
\beq
\label{exact-sol}
m(\xi) =
\left\{ 
\begin{array}{ll}
	m_+ e^{\lambda \xi} (1-e^{-\xi})^{\lambda - b}, & \quad  \xi>0,\\
	m_- e^{\lambda \xi} (1-e^{\xi})^{-\lambda -b}, & \quad \xi<0, \end{array}\right.
\eeq
where $m_-$ and $m_+$ are arbitrary constants. The eigenfunction $v$ is recovered from solutions of the second-order equation 
\beq
\label{exact-eq-v}
v(\xi) - v''(\xi) = 
\left\{ 
\begin{array}{ll}
	m_+ e^{\lambda \xi} (1-e^{-\xi})^{\lambda - b}, & \quad  \xi>0,\\
	m_- e^{\lambda \xi} (1-e^{\xi})^{-\lambda -b}, & \quad \xi<0. \end{array}\right.
\eeq
If $v \in L^2(\mathbb{R})$, then $m \in H^{-2}(\mathbb{R})$, the dual space of $H^2(\mathbb{R})$. Since $m(\xi) \sim m_+ e^{\lambda \xi}$ as $\xi \to +\infty$, 
then $v \in L^2(\mathbb{R})$ for $\realpart{\lambda} \geq 0$ exists with 
the only choice $m_+ = 0$ and $v(\xi) = c_+ e^{-\xi}$, where $c_+$ is arbitrary. Similarly, $m(\xi) \sim m_- e^{\lambda \xi}$ as $\xi \to -\infty$. 
If $\realpart{\lambda} = 0$, then $v \in L^2(\mathbb{R})$ exists with 
the only choice $m_- = 0$ and $v(\xi) = c_- e^{\xi}$, where $c_-$ is arbitrary. In this case, $v = c_1 \phi + c_2 \phi'$ with $c_1 \pm c_2 = c_{\mp}$, which corresponds to the double eigenvalue $\lambda = 0$ in $\sigma_p(L)$. 

It remains to consider the case $m_+ = 0$, $m_- \neq 0$, and $\realpart{\lambda} > 0$. With the normalization $m_- = 1$, the solution of (\ref{exact-eq-v}) for $\xi < 0$ can be written in the form 
\beq
\label{exact-sol-v}
v(\xi) = e^{\lambda \xi} f(\xi) (1-e^{\xi})^{2-\lambda-b}, \quad \xi < 0,
\eeq
where $f(\xi)$ satisfies the second-order differential equation
\begin{eqnarray}
\nonumber
&& (1-e^{\xi})^2 (f'' + 2(2-b) f' + (b-1)(b-3) f) \\
&& + (\lambda + b - 2)  (1-e^{\xi}) (2f' + (3-2b) f) 
+ (\lambda + b - 2)(\lambda + b -1) f = -1.
\label{exact-sol-F}
\end{eqnarray}
The homogeneous part of the second-order (\ref{exact-sol-F}) with the regular singular point $\xi = 0$ is associated with the indicial equation 
$$
\sigma^2 + (2 \lambda + 2 b - 5) \sigma + (\lambda + b - 2) (\lambda+ b -1) = 0
$$
for power solutions $f(\xi) \sim \xi^{\sigma}$. If $\lambda + b \neq \{1,2\}$, then $0$ is not among the roots of the indicial equation, whereas if $\lambda + b = \{1,2\}$, then $0$ is a simple root of the indicial equation. 
By the Frobenius theory (see, e.g., Chapter 4 in \cite{Teschl}), there exists a particular solution to the differential equation (\ref{exact-sol-F}) with the following behavior near the regular singular point $\xi = 0$:
\beq
f_p(\xi) \sim \left\{ \begin{array}{ll} 1 + \mathcal{O}(|\xi|), & \quad \lambda + b \neq \{1,2\}, \\
		\log|\xi| + \mathcal{O}(|\xi| \log|\xi|), & \quad \lambda + b = \{1,2\},
\end{array}
\right.
\quad \mbox{\rm as} \quad 
\xi \to 0^-,
\eeq
which yields the corresponding behavior of {{$v_p(\xi)$}} from (\ref{exact-sol-v}):
\beq
{{v_p}}(\xi) \sim \left\{ \begin{array}{ll} |\xi|^{2-b-\lambda}, & \quad \lambda + b \neq \{1,2\}, \\
	\log|\xi|, & \quad \lambda + b = 2, \\
	|\xi| \log|\xi|, & \quad \lambda + b = 1, 
\end{array}
\right.
\quad \mbox{\rm as} \quad 
\xi \to 0^-.
\eeq
Hence, {{\eqref{exact-eq-v} has a solution}} $v \in L^2(\mathbb{R})$ if and only if $\realpart{\lambda}+b-2< \frac{1}{2}$, for $0 < \realpart{\lambda} < \frac{5}{2} - b$. 

Summarizing and using the symmetry above, $\sigma_p(L)$ exists if $b < \frac{5}{2}$ and is located for $0 < |\realpart{\lambda}| <  \frac{5}{2}-b$.
\end{proof}

\begin{remark}
	In the case of Camassa--Holm equation $(b=2)$, there exists the exact solution of the differential equation (\ref{exact-eq-v}) with $\lambda \neq \{-1,0,1\}$ in the form 
	\begin{equation}
	v(\xi) = \frac{1}{\lambda (1-\lambda^2)} \left\{ 
	\begin{array}{ll} 
		\displaystyle	m_+ (\lambda + e^{-\xi}) (e^{\xi}-1)^{\lambda}, & \quad \xi > 0, \\
	\displaystyle	m_- (\lambda - e^{\xi}) (e^{-\xi}-1)^{-\lambda}, & \quad \xi < 0.
	\end{array} \right.
	\end{equation}
	It corresponds to the exact solution 
	$$
	f(\xi) = \frac{(\lambda - e^{\xi})}{\lambda (1 - \lambda^2)}, \quad \xi < 0,
	$$
	of equation (\ref{exact-sol-F}) for $b = 2$.
	If $m_- \neq 0$, then $v(\xi) \sim e^{\lambda \xi}$ as $\xi \to -\infty$ and $v(\xi) \sim |\xi|^{-\lambda}$ as $\xi \to 0^-$  so that $v \in L^2(\mathbb{R})$ if $0 < \realpart{\lambda} < \frac{1}{2}$.
\end{remark}

\subsection{Proof of Theorem \ref{S}}

We check assumptions of Theorem 1 in \cite{GP2} that the intersections $\sigma_p(L_0) \cap \rho(L)$ and $\sigma_p(L) \cap \rho(L_0)$ are empty. 

For $b < \frac{5}{2}$, $\sigma_p(L_0)$ in Theorem \ref{S0} consists of 
the bands $0 < |\realpart{\lambda}| < \frac{5}{2}-b$, whereas 
$\sigma_p(L)$ in Lemma \ref{im} consists of the same bands and an additional double zero 
eigenvalue. Since the resolvent set $\rho(L_0)$ in Theorem \ref{S0} 
consiste of the bands $|\realpart{\lambda}| > \frac{5}{2}-b$, 
$\sigma_p(L) \cap \rho(L_0)$ is an empty set. Since $\sigma_p(L_0) \subset \sigma_p(L)$,  $\sigma_p(L_0) \cap \rho(L)$ is also an empty set.

For $b \geq \frac{5}{2}$, $\sigma_p(L_0)$ in Theorem \ref{S0} is an empty set, 
hence $\sigma_p(L_0) \cap \rho(L)$ is an empty set, whereas 
$\sigma_p(L) = \{0\}$ does not belong to the resolvent set of $L_0$ 
in the bands  $|\realpart{\lambda}| > b -\frac{5}{2}$, hence 
$\sigma_p(L) \cap \rho(L_0)$ is also an empty set. 

Since $Q$ is a compact operator in $L^2(\mathbb{R})$ by Lemma \ref{lem-compact}, 
Theorem 1 in \cite{GP2} suggests that $\sigma(L) = \sigma(L_0)$. The proof of Theorem \ref{S} repeats the statement of Theorem \ref{S0} with additional 
information about the double zero eigenvalue. 

\section{Time evolution of the linearized system}
\label{sec-time}

Due to the difference in the spectral properties of the linearized operator $L$ in $L^2(\mathbb{R})$ for $b < \frac{5}{2}$ and $b > \frac{5}{2}$ in Theorem \ref{S}, the growth of the $L^2$ norm of the peaked perturbations is different due to the point spectrum for $b < \frac{5}{2}$ and due to the residual spectrum for $b > \frac{5}{2}$, as follows from the balance equation (\ref{balance-eq}). 
We illustrate this difference here by studying exact solutions to the initial-value problem 
\beq
\label{ivp}
\left\{ \begin{array}{ll} 
	v_t = (1-\phi) v_{\xi} + (2-b) \phi' v, & \quad t > 0, \\
	v |_{t=0} = v_0, & \quad
	\end{array} 
\right.
\eeq
where $v_0 \in {\rm Dom}(L)$.
The initial-value problem coincides with the linearized equation 
for the truncated operator $L_0$. By Theorem \ref{S0}, the unstable spectrum of the operator $L_0$ is the point spectrum for $b < \frac{5}{2}$ and the residual spectrum for $b > \frac{5}{2}$. 

The initial-value problem (\ref{ivp}) can be solved exactly by using the method of characteristics also used in \cite{Natali,MP-2021,GP1}. The following lemma gives the bounds obtained from the exact solution of the initial-value problem (\ref{ivp}).

\begin{lemma}
	\label{lem-final}
	For every $v_0 \in {\rm Dom}(L)$, the initial-value problem (\ref{ivp}) admits the unique solution $v \in C(\mathbb{R},{\rm Dom}(L))$ satisfying the following properties:\\
	\begin{itemize}
		\item If $b = \frac{5}{2}$, then $\| v(t,\cdot) \|_{L^2} = \| v_0 \|_{L^2}$. \\
		\item If $b > \frac{5}{2}$, then $\| v(t,\cdot) \|_{L^2(-\infty,0)} \leq \| v_0 \|_{L^2(-\infty,0)}$ and there is a positive constant $C_0$ that depends on $v_0$ such that $\|  v(t,\cdot) \|_{L^2(0,\infty)} \geq C_0 e^{(b-\frac{5}{2})t}$. \\
		\item If $b < \frac{5}{2}$, then $\| v(t,\cdot) \|_{L^2(0,\infty)} \leq \| v_0 \|_{L^2(0,\infty)}$ and there exists $v_0 \in {\rm Dom}(L)$ such that $\| v(t,\cdot) \|_{L^2(-\infty,0)} = e^{\lambda_0 t} \| v_0 \|_{L^2(-\infty,0)}$ with $\lambda_0 \in (0,\frac{5}{2}-b)$. 
	\end{itemize}
\end{lemma}

\begin{proof}
By using the method of characteristics, we introduce 
the characteristic curves $\xi = X(t,s)$ satisfying $X_t = \phi(X) - 1$ 
with $X |_{t=0} = s$. The exact solution for $X(t,s)$ is readily available 
(see, e.g. \cite{Natali}): 
\beq
X(t,s) = \left\{ \begin{array}{ll} 
	\log\left[ 1 + (e^s-1) e^{-t} \right], & \quad s > 0, \\
	-\log\left[ 1 + (e^{-s}-1) e^{t} \right], & \quad s < 0.
\end{array} \right.
\eeq
Along the characteristic curves, the function $V(t,s) := v(t,\xi = X(t,s))$ 
satisfies 
\beq
\left\{ \begin{array}{ll}
V_t = (2-b) \phi'(X(t,s)) V, & \quad t > 0, \\
V|_{t=0} = v_0(s), 
\end{array} \right.
\eeq
which is uniquely solved separately for $s > 0$ and $s < 0$ by 
\beq
V(t,s) = v_0(s) \left[ \frac{\partial X}{\partial s} \right]^{2-b} = 
\left\{ \begin{array}{ll} 
	v_0(s) \left[ 1 + (e^t-1) e^{-s} \right]^{b-2}, & \quad s > 0, \\
	v_0(s) \left[ 1 + (e^{-t}-1) e^{s} \right]^{b-2}, & \quad s < 0.
\end{array} \right.
\eeq
By the chain rule, we obtain 
\beq
\label{solution-right}
\| v(t,\cdot) \|^2_{L^2(0,\infty)} = \int_0^{\infty} |v_0(s)|^2 \left[ 1 + (e^t-1) e^{-s} \right]^{2b-5} ds
\eeq
and 
\beq
\label{solution-left}
\| v(t,\cdot) \|^2_{L^2(-\infty,0)} = \int_{-\infty}^0 |v_0(s)|^2 \left[ 1 + (e^{-t}-1) e^{s} \right]^{2b-5} ds.
\eeq

If $b = \frac{5}{2}$, the linear evolution is $L^2$-preserving in time both for $s > 0$ and $s < 0$.

If $b > \frac{5}{2}$, then it follows from (\ref{solution-right}) and (\ref{solution-left}) that
$$
\| v(t,\cdot) \|^2_{L^2(0,\infty)} \geq (e^t - 1)^{2b-5} \int_0^{\infty} |v_0(s)|^2 e^{-(2b-5)s} ds, \quad 
\| v(t,\cdot) \|^2_{L^2(-\infty,0)} \leq \int_{-\infty}^0 |v_0(s)|^2 ds.
$$
This proves that the $L^2$ norm of the perturbation grows exponentially in time  on the right side of the peak at $\xi = 0$, while that on the left side of the peak  stays bounded in time. Moreover, the growth rate of the $L^2$ norm is exactly $e^{(b-\frac{5}{2})t}$.

If $b < \frac{5}{2}$, then it follows from (\ref{solution-right}) and (\ref{solution-left}) that
$$
\| v(t,\cdot) \|^2_{L^2(0,\infty)} \leq \int_0^{\infty} |v_0(s)|^2 ds, 
$$
so that the $L^2$ norm of the perturbation  stays bounded in time on the right side of the peak at $\xi = 0$. It may first look like the $L^2$ norm 
on the left side of the peak stays bounded too with the formal limit
\beq
\label{limit-integral}
\lim_{t \to +\infty} \| v(t,\cdot) \|^2_{L^2(-\infty,0)} = \int_{-\infty}^0 |v_0(s)|^2 (1-e^s)^{2b-5} ds.
\eeq
However, the limit may not exist if $v_0(0) \neq 0$. If $v_0$ is the eigenfunction of the point spectrum of $L_0$ for the eigenvalue $\lambda_0 \in (0,\frac{5}{2}-b)$, which follows from (\ref{Lv0S}), that is, 
\beq
\label{eigenfunction}
v_0(s) =\frac{e^{\lambda_0 s}}{(1-e^{s})^{\lambda_0 + b-2}}, \quad s < 0,
\eeq
then after elementary transformations, we obtain for $s < 0$
$$
V(t,s) = v_0(s) \left[ 1 + (e^{-t}-1) e^{s} \right]^{b-2} = 
\frac{e^{\lambda_0 t + \lambda_0 X(t,s)}}{(1-e^{X(t,s)})^{\lambda_0 + b-2}} = v_0(X(t,s)) e^{\lambda_0 t},
$$
so that $\| v(t,\cdot) \|_{L^2(-\infty,0)} = \| v_0 \|_{L^2(-\infty,0)} e^{\lambda_0 t}$ grows exponentially in time. Due to the weak singularity of $v_0(s)$ as $s \to 0^-$, which is allowed in ${\rm Dom}(L) \subset L^2(\mathbb{R})$, the integral of $|v_0(s)|^2 (1-e^s)^{2b-5}$ in (\ref{limit-integral}) diverges. 
\end{proof}

\begin{remark}
	The results of Lemma \ref{lem-final} are clearly related to the spectral properties of $L_0$ in Theorem \ref{S0}. In more details, we point the following:
	\begin{itemize}
		\item If $b = \frac{5}{2}$, the preservation of the $L^2$ norm of the solution $v \in C(\mathbb{R},{\rm Dom}(L))$ is related to $\sigma(L_0) = \{ i \mathbb{R}\}$.
	\item If $b > \frac{5}{2}$, the growth rate $e^{(b-\frac{5}{2})t}$ of the $L^2$ norm of the solution $v \in C(\mathbb{R},{\rm Dom}(L))$  agrees with the width of the unstable strip at $0 < |\realpart{\lambda}| \leq b-\frac{5}{2}$.
		\item If $b < \frac{5}{2}$, the instability is obtained from the eigenfunction (\ref{eigenfunction}) corresponding to the eigenvalue of the point spectrum in the unstable strip at $0 < |\realpart{\lambda}| \leq \frac{5}{2}-b$.  
		If the initial condition satisfies $v_0(0) = 0$ and the integral of $|v_0(s)|^2 (1-e^s)^{2b-5}$ in (\ref{limit-integral}) converges, the $L^2$ norm of the solution $v \in C(\mathbb{R},{\rm Dom}(L))$ is bounded for all times.
		\end{itemize}
\end{remark}

\begin{remark}
	Exact solutions for the initial-value problem associated with the full linearized operator $L$ can be obtained by the method of characteristics if $b = 2$ \cite{Natali}. The growth of perturbations in $H^1$ norm was obtained in \cite{Natali} by explicit computations. Since perturbations were considered in $H^1(\mathbb{R}) \cap W^{1,\infty}(\mathbb{R})$ with the additional condition $v_0(0) = 0$, the $L^2$ norm of such perturbations does not grow in time and the eigenfunctions (\ref{eigenfunction}) are excluded from the choice of the initial condition.
\end{remark}

\appendix

\section{Antisymmetric solution $v_b$ to $L^*v=0$ in the case $b > 3$}
\label{AppA}

We consider the equation
$$
L^*v=0
$$
where $L^*$ is defined in \eqref{Ls}. We make the substitution $v=w-w_{\xi\xi}$ and assume that $w$ and $w_{\xi\xi}$ are bounded and continuous functions. By using integration by parts to the last two terms in \eqref{Ls} applied to $w_{\xi\xi}$, we obtain the following differential equation
$$
 (1-\phi) w_{\xi\xi\xi} + (b-3) \phi' w_{\xi\xi} + (2(b-1)\phi-1)w_\xi=0.
$$
There exists two particular solutions of this differential equation:
$$
w_{\rm sym} = (1-\phi)^{b-1}, \quad 
w_{\rm anti} = {\rm sgn} (1-\phi)^{b-1},
$$
where ${\rm sgn}(\xi) = 1$ if $\xi > 0$ and ${\rm sgn}(\xi) = -1$ if $\xi < 0$. 
The assumption that $w$ and $w_{\xi\xi}$ are bounded and continuous functions is satisfied only if $b > 3$. 
Applying $(1-\partial_{\xi}^2)$ to $w_{\rm anti}$, we  
compute the antisymmetric solution $v_b$ to $L^*v = 0$:
$$
 v_b = {\rm sgn} (1-\phi)^{b-3}\left(b(b-2)\phi^2+(3-b)\phi-1\right).
$$
The solution $v_b$ is bounded and continuous at $\xi = 0$ if $b > 3$. There exists also 
the symmetric solution to $L^* v = 0$ obtained by applying $(1-\partial_{\xi}^2)$ to $w_{\rm sym}$. 
The solution is not used because $L^* v = 0$ is also satisfied by the symmetric, bounded, and continuous function $v = 1$.

\bibliographystyle{unsrt}

\begin{thebibliography}{99} 
	
\bibitem{Buhler18} T.~B\"uhler and D.~A.~Salamon, {\sl{Functional analysis}}, Graduate
Studies in Mathematics, {\bf{191}}, American Mathematical Society,
Providence, RI, 2018.

\bibitem{Cam} R. Camassa, D. Holm, ``An integrable shallow water equation with peaked solitons", Phys. Rev. Lett. {\bf 71} (1993), 1661--1664.

\bibitem{Cam2} R. Camassa, D. Holm and J. Hyman, ``A new integrable shallow water equation", Adv. Appl. Mech. {\bf 31} (1994), 1--33.

\bibitem{Char1} E. G. Charalampidis, R. Parker, P. G. Kevrekidis, and S. Lafortune, ``The stability of the $b$-family of peakon equations'', arXiv:2012.13019 (2020).

\bibitem{Const4} A. Constantin and D. Lannes, ``The hydrodynamical relevance of the Camassa–Holm and Degasperis–Procesi equations", Arch. Ration. Mech. Anal. {\bf 192} (2009), 165--186.

\bibitem{Const5} A. Constantin and L. Molinet, ``Orbital stability of solitary waves for a shallow water equation", Physica D {\bf 157} (2001), 75--89.

\bibitem{Cons1} A. Constantin and W. A. Strauss, ``Stability of peakons", Comm. Pure Appl. Math. {\bf 53} (2000), 603--610.

\bibitem{CS-02} A. Constantin and W.A. Strauss, ``Stability of the Camassa--Holm solitons”, J. Nonlinear Sci. {\bf 12} (2002), 415--422.

\bibitem{dhh} A.~Degasperis, D. D.~Holm and A. N. W.~Hone, 
``A new integrable equation with peakon solutions", 
Theor. and Math. Phys. {\bf 133} (2002) 1461--1472.

\bibitem{dp} A. Degasperis and  M. Procesi, ``Symmetry and Perturbation Theory", in {\em Asymptotic Integrability} (A. Degasperis and G. Gaeta, editors) 
(World Scientific Publishing, Singapore, 1999), 23--37.

\bibitem{Dullin} H. R. Dullin, G. A. Gottwald, and D. D. Holm, ``An integrable shallow water equation with linear and nonlinear
dispersion", Phys. Rev. Lett. {\bf 87} (2001) 194501.

\bibitem{DS} N. Dunford, J. T. Schwartz, {\em Linear Operators. Part II: Spectral Theory} (John Wiley \& Sons, New York, 1963)

\bibitem{Fokas} B. Fuchssteiner and A. S. Fokas, ``Symplectic structures, their Backlund transformations and hereditary symmetries", Physica D {\bf 4} (1981) 47--66.

\bibitem{GP1} A. Geyer and D. E. Pelinovsky, ``Linear instability and uniqueness of the peaked periodic wave
in the reduced Ostrovsky equation",  SIAM J. Math. Anal. {\bf 51} (2019), 1188--1208.

\bibitem{GP2} 
A.~Geyer and D. E. ~Pelinovsky, %
Proceedings of the American Mathematical Society {\bf 148} (2020)  5109--5125.


\bibitem{Molinet} Z. Guo, X. Liu, L. Molinet, and Z. Yin, ``Norm inflation of
the Camassa-Holm and relation equations in the critical space", J. Diff. Eqs. {\bf 266} (2019), 1698--1707.



\bibitem{Him} A. Himonas, K. Grayshan, and C. Holliman, ``Ill-posedness for the b-family of equations", J. Nonlin. Sci. {\bf 26} (2016), 1175--1190.


\bibitem{Holm1} D. D. Holm and M. F. Staley, ``Nonintegrability of a fifth-order equation with integrable two-body dynamics", Phys. Lett. A {\bf 308} (2003) 437--444.

\bibitem{Holm2} D. D. Holm and M. F. Staley, ``Wave structure and nonlinear balances in a family of evolutionary PDEs", SIAM J. Appl. Dyn. Syst. {\bf 2} (2003) 323--380.

\bibitem{Hone14} A. N. W.~Hone and S.~Lafortune, ``Stability of stationary solutions for nonintegrable peakon equations'',
Physica D {\bf 269} (2014), 28--36.



\bibitem{Johnson} R. S. Johnson, ``Camassa–Holm, Korteweg–de Vries and related models for water waves", J. Fluid Mech. {\bf 455} (2002), 63--82.

\bibitem{Liu-21} J. Li, Y. Liu, and Q. Wu, ``Spectral stability of smooth solitary waves for the Degasperis--Procesi equation", 
J. Math. Pures Appl. {\bf 142} (2020) 298--314.

\bibitem{LinLiu} Z. Lin and Y. Liu, ``Stability of peakons for the Degasperis--Procesi equation",
Comm. Pure Appl. Math. {\bf 62} (2009), 125--146.


\bibitem{Linares} F. Linares, G. Ponce, and Th. C. Sideris, ``Properties of solutions to the Camassa--Holm
equation on the line in a class containing the peakons", Advanced Studies in Pure Mathematics {\bf 81} (2019), 196--245.

\bibitem{MP-2021} A. Madiyeva and D. E. Pelinovsky, `` Growth of perturbations to the peaked periodic waves in the Camassa-Holm equation", SIAM J. Math. Anal. (2021), in print.


\bibitem{Natali} F. Natali and D. E. Pelinovsky, ``Instability of $H^1$-stable peakons in the Camassa-Holm equation", J. Diff. Eqs. {\bf 268} (2020), 7342--7363.

\bibitem{Renardy} 
M.~Renardy and R.~C.~Rogers, %
{\sl{An Introduction to Partial Differential Equations}}, 
Texts in Applied Mathematics, Springer-Verlag, 2nd edition (2004).

\bibitem{Teschl} G. Teschl, 
{\em Ordinary Differential Equations and Dynamical Systems},
Graduate Studies in Mathematics {\bf 140} 
(AMS, Providence RI, 2012).
\end{thebibliography}

\end{document}